\documentclass[a4paper,11pt,reqno]{amsart}

\usepackage{lscape}
\usepackage{multirow}
\usepackage[all]{xy}
\usepackage{amsmath,amssymb,graphicx,enumerate,amsthm}

\def\id{\mathrm{id}}

\def\<#1,#2>{\langle\,#1,\,#2\,\rangle}

\def\inf{\mathrm{inf}}

\def\sigm{{\mathrm{\Sigma M}}}

\def\qed{\ensuremath{\hfill\Box}}

\newcommand{\intersect}{\mathinner{\cap}}

\def\Aut{\mathrm{Aut}}
\def\so{\mathfrak{so}}
\def\SO{\mathrm{SO}}

\def\Spin{\mathrm{Spin}}

\def\GL{\mathrm{GL}}

\def\1{\mathbf{1}}
\def\#{\sharp}

\def\Aut{\mathrm{Aut}}

\def\e{\varepsilon}

\def\id{\mathrm{id}}

\def\l{\lambda}

\def\R{\mathbb{R}}

\def\sgn{\mathrm{sgn}}

\def\wb{Weitzenb\"ock }

\def\<#1,#2>{\langle\,#1,\,#2\,\rangle}

\def\rectangle(#1,#2)[#3,#4]#5{
 \multiput(#1,#2)(#3,0)2{\line(0,1){#4}}\multiput(#1,#2)(0,#4)2{\line(1,0){#3}}
 \put(#1,#2){\vbox to #4pt{\hbox to #3pt{\hfill}\vfill}}}
\def\recttext(#1,#2)[#3,#4]#5{\put(#1,#2)
 {\vbox to #4pt{\vfill\hbox to #3pt{\hss#5\hss}\vfill}}}

\def\qed{\ensuremath{\hfill\Box}}
\newtheorem{Lemma}{Lemma}[section]
\newtheorem{Proposition}[Lemma]{Proposition}
\newtheorem{Theorem}[Lemma]{Theorem}

\theoremstyle{definition}
\newtheorem{Definition}[Lemma]{Definition}
\newtheorem{Example}[Lemma]{Example}
\newtheorem{Question}[Lemma]{Question}
\newtheorem{Remark}[Lemma]{Remark}

\newcommand{\ip}[1]{\langle#1\rangle}
\newcommand{\NE}{\mathcal{N\!E}}

\def\vs{\vskip .4cm}

\numberwithin{equation}{section}

\title[A New Proof of Branson's Classification of Elliptic Generalized Gradients]{A New Proof of Branson's Classification\\ of Elliptic Generalized Gradients}
\author{Mihaela Pilca}
\thanks{The author gratefully ackowledges partial financial support from DFG-Graduate School 1269 ``Global Structures in Geometry and Analysis" and SFB/TR 12 ``Symmetries and Universality in Mesoscopic Systems".}
\address{Mihaela Pilca \\ Mathematisches Institut\\ Universit\"at zu K\"oln\\ Weyertal 86-90 D-50931 K\"oln\\ Germany and Institute of Mathematics ``Simion Stoilow" of the Romanian Academy\\ 21 Calea Grivitei Str.\\ 010702-Bucharest\\ Romania.}
\email{mpilca@mi.uni-koeln.de}

\begin{document}

\begin{abstract}
We give a representation theoretical proof of Branson's classification, \cite{br1}, of mini\-mal elliptic sums of generalized gradients. The original proof uses tools of harmonic analysis, which as powerful as they are, seem to be specific for the structure groups $\SO(n)$ and $\Spin(n)$. The different approach we propose is based on the relationship between ellipticity and optimal Kato constants and on the representation theory of $\so(n)$. Optimal Kato constants for elliptic operators were computed by Calderbank, Gauduchon and Herzlich, \cite{cgh}. We extend their method to all generalized gradients (not necessarily elliptic) and recover Branson's result, up to one special case. The interest of this method is that it is better suited to be applied for classifying elliptic sums of generalized gradients of $G$-structures, for other subgroups $G$ of the special orthogonal group.
 
\vs

\end{abstract}

\maketitle

\section{Introduction}

The classical notion of generalized gradients, also called Stein-Weiss operators, was first introduced by Stein and Weiss, \cite{sw}, on an oriented Riemannian manifold, as a generalization of the Cauchy-Riemann equations. They are first order differential operators acting on sections of vector bundles associated to irreducible representations of the special orthogonal group (or of the spin group, if the manifold is spin), which are given by the following universal construction: one projects onto an irreducible subbundle the covariant derivative induced on the associated vector bundle by the Levi-Civita connection or, more generally, by any metric connection.

Some of the most important first order differential operators which naturally appear in geometry are, up to normalization, particular cases of generalized gradients. For example, on a Riemannian manifold, the exterior differential acting on differential forms, its formal adjoint, the codifferential, and the conformal Killing operator on $1$-forms are generalized gradients. On a spin manifold, classical examples of generalized gradients are the Dirac operator, the twistor (or Penrose) operator and the Rarita-Schwinger operator. 

On an oriented Riemannian manifold, generalized gradients naturally give rise, by composition with their formal adjoints, to second order differential operators acting on sections of associated vector bundles. Particularly important are the extreme cases of linear combinations of such second order operators: if the linear combination provides a zero-order operator, then it is a curvature term and one obtains a so-called \wb formula; if the \mbox{linear} combination is a second order differential operator, then it is interesting to determine when it is elliptic. Whereas \wb formulas play a key role in relating the local differential geometry to global topological properties by the so-called Bochner method (for recent systematic approaches to the description of all \wb formulas we refer to \cite{h3} and \cite{usgw2}), the importance of elliptic operators is well established, see \emph{e.g.} the seminal paper \cite{as}.

The elliptic second order differential operators constructed this way from generali\-zed gradients between vector bundles with structure group $\SO(n)$ or $\Spin(n)$ over an oriented Riemannian manifold were completely classified by Branson, \cite{br1}. The classical example is the Laplacian acting on differential forms, which is obtained by assembling two generalized gradients, namely the exterior differential and the codifferential. Branson showed that it is enough to take surprisingly few generalized gradients in order to obtain an elliptic operator. It turned out that Laplace-type operators represent the generic case. Namely, apart from a few known exceptions, each minimal elliptic operator is given by a pair of generalized gradients. The arguments used by Branson are based on techniques of harmonic analysis and explicit computations of the spectra of generalized gradients on the sphere. Partial results were previously obtained by Kalina, Pierzchalski and Walczak, \cite{kpw}, who showed that the only generalized gradient which is strongly elliptic is given by the projection onto the Cartan summand. Furthermore, the projection onto the complement of the Cartan summand is also elliptic, by a result of Stein and Weiss, \cite{sw}.

In this paper we give a new proof of Branson's classification of such elliptic operators. The method we use is completely different from the original one in \cite{br1}, which seems to be specific for the two structure groups $\SO(n)$ or $\Spin(n)$. Our approach is mainly based on the one hand on the relationship between ellipticity and Kato constants and on the other hand on the representation theory of the Lie algebra $\mathfrak{so}(n)$. The starting point is the remark that these elliptic operators are closely related to the existence of refined Kato inequalities, which was first noticed by Bourguignon, \cite{jpb}. The explicit computation of the optimal Kato constants for all elliptic differential operators obtained from generalized gradients by the above construction was given by Calderbank, Gauduchon and Herzlich, \cite{cgh}. In the first part of our proof we extend their computation to all (not necessarily elliptic) sums of generalized gradients and then use it to recover Branson's list of minimal elliptic operators, up to an exceptional case. In the second part of the proof we show that these are \emph{all} minimal elliptic operators, \emph{i.e.} to find the list of maximal non-elliptic operators. The tool used here is the branching rule for the special orthogonal group.

The construction of the classical generalized gradients for the structure groups $\SO(n)$ and $\Spin(n)$ can be carried over to $G$-structures, when there is a reduction of the structure group of the tangent bundle of a Riemannian manifold $(M^n,g)$ to a closed subgroup $G$ of $\SO(n)$ (see \emph{e.g.} \cite{mpth}). The arguments of our new approach suggest that they should carry over to other subgroups $G$ of $\SO(n)$, in order to provide the classification of natural elliptic operators constructed from $G$-generalized gradients. 

The paper is organized as follows. We first present in \S~\ref{sectgengrad} the general setting and notation. We then recall in \S~\ref{brclass}  Branson's classification of minimal elliptic operators that naturally arise from generalized gradients (see Theorems~\ref{brclassodd} and ~\ref{brclassev}). In the last section we give our new proof of Branson's classification. In \S~\ref{ellkato} we extend the computation of the Kato constant given in \cite{cgh} to all sums of generalized gradients and then recover the list of minimal elliptic operators. In \S~\ref{nonellbr} we show how the branching rules for the special orthogonal group let us conclude that this list is complete.

\section{Generalized Gradients}\label{sectgengrad}

We briefly recall in this section the general construction of generalized gradients given by Stein and Weiss, \cite{sw}, on an oriented Riemannian (spin) manifold. 

Let us first state the general context, fix the notation and briefly recall the representation theoretical background needed to define the generalized gradients. The description of the representations of $\so(n)$, the Lie algebra of $\SO(n)$, differs slightly according to the parity of $n$. We write $n=2m$ if $n$ is even and $n=2m+1$ if $n$ is odd, where $m$ is the rank of $\so(n)$. Let $\{e_1, \dots,e_n\}$ be a fixed oriented orthonormal basis  of $\mathbb{R}^n$, so that $\{e_i\wedge e_j\}_{i<j}$ is a basis of the Lie algebra $\so(n)\cong\Lambda^2\mathbb{R}^n$. We also fix a Cartan subalgebra $\mathfrak{h}$ of $\so(n)$ by the basis $\{e_1\wedge e_2, \dots, e_{2m-1}\wedge e_{2m}\}$ and denote the dual basis of $\mathfrak{h}^*$ by $\{\varepsilon_1,\dots,\varepsilon_m\}$. The Killing form is normalized such that this basis is orthonormal. Roots and weights are given by their coordinates with respect to the orthonormal basis $\{\varepsilon_i\}_{i=\overline{1,m}}$. Finite-dimensional complex irreducible $\so(n)$-representations are parametrized by the \emph{dominant weights}, \emph{i.e.} those weights whose coordinates are either all integers or all half-integers,\linebreak $\lambda=(\lambda_1,\dots,\lambda_m)\in\mathbb{Z}^m\cup (\frac{1}{2}+\mathbb{Z})^m$ and which satisfy the inequality:
\begin{equation}\label{domw}
\begin{split}
\lambda_1\geq \lambda_2 \geq \cdots \lambda_{m-1}\geq |\lambda_m|, \quad &\text{if }n=2m, \text{ or }\\
\lambda_1\geq \lambda_2 \geq \cdots \lambda_{m-1}\geq \lambda_m \geq 0, \quad &\text{if }n=2m+1.
\end{split}
\end{equation}
Through this parametrization a dominant weight $\lambda$ is the highest weight of the corresponding representation. With a slight abuse of notation, we use the same symbol for an irreducible representation and its highest weight. The representations of $\so(n)$ are in one-to-one correspondence with the representations of the corresponding simply-connected Lie group, \emph{i.e.} $\Spin(n)$, the universal covering of $\SO(n)$. The representations which factor through $\SO(n)$ are exactly those with $\lambda\in\mathbb{Z}^m$. For example, the (complex) standard representation, denoted by $\tau$, is given by the weight $(1,0,\dots,0)$; the weight $(1,\dots,1,0,\dots,0)$ (with $p$ ones) cor\-responds to the $p$-form representation $\Lambda^p\mathbb{R}^n$, whereas the dominant weights $\lambda=(1,\dots,1,\pm1)$, for $n=2m$, correspond to the representation of selfdual, respectively antiselfdual $m$-forms; the representation of totally symmetric traceless tensors $S^p_0\mathbb{R}^n$ has highest weight $(p,0,\dots,0)$. 

The following so-called \emph{classical selection rule} (see \cite{feg}) describes the decomposition of the tensor product $\tau\otimes\lambda$ into irreducible $\so(n)$-representations, where $\tau$ is the standard representation and $\lambda$ is any irreducible representation. 
\begin{Lemma}\label{selectrule}
An irreducible representation of highest weight $\mu$ occurs in the decomposition of $\tau\otimes\lambda$ if and only if the following two conditions are fulfilled:
\begin{itemize}
	\item[(i)] $\mu=\lambda\pm\varepsilon_j$, for some $j=1,\dots, m$, or $n=2m+1$, $\lambda_m>0$ and $\mu=\lambda$,	
  \item[(ii)] $\mu$ is a dominant weight, \emph{i.e.} satisfies \eqref{domw}.
\end{itemize}
\end{Lemma}
We adopt the same terminology as in \cite{usgw2} and call \emph{relevant weights of $\lambda$} (and write $\varepsilon\subset\lambda$) those weights $\varepsilon$ of $\tau$, $\varepsilon\in\{0,\pm\varepsilon_1,\dots,\pm\varepsilon_m\}$, with the property that $\lambda+\varepsilon$ occurs in the decomposition of $\tau\otimes\lambda$. The decomposition of the tensor product is then expressed as follows:
\begin{equation}\label{decompirred}
\tau\otimes\lambda=\underset{\varepsilon\subset\lambda}{\oplus}(\lambda+\varepsilon).
\end{equation}
The essential property of the decomposition \eqref{decompirred} is that it is multiplicity-free, \emph{i.e.} the isotypical components are actually irreducible. It thus follows that the projections onto each irreducible summand $\lambda+\varepsilon$ in the splitting are well-defined; we denote them by $\Pi_\varepsilon$.

Let $(M,g)$ be an oriented Riemannian manifold, $\SO_g M$ denotes the principal $\SO(n)$-bundle of oriented orthonormal frames and $\nabla$ any metric connection, considered either as a connection $1$-form on $\SO_g M$ or as a covariant derivative on the tangent bundle $\mathrm{TM}$. If $M$ has, in addition, a spin structure, then we consider the corresponding principal $\Spin(n)$-bundle, denoted by $\Spin_g M$, and the induced metric connection $\nabla$. We consider vector bundles associated to $\SO_g M$ (or $\Spin_g M$) and irreducible $\SO(n)$ (or $\Spin(n)$)-representations of highest weight $\lambda$ and denote them by $V_\lambda M$. For instance, the tangent bundle is associated to the standard representation $\tau:\SO(n)\hookrightarrow \GL(\mathbb{R}^n)$ and the bundle of $p$-forms is associated to the irreducible representation of highest weight $\lambda=(1,\dots,1,0,\dots,0)$ (with $p$ ones). We identify the cotangent bundle $\mathrm{T}^*\mathrm{M}$ and the tangent bundle $\mathrm{TM}$ using the metric $g$, since they are associated to equivalent $\SO(n)$-representations. The decomposition \eqref{decompirred} carries over to the associated vector bundles:
\begin{equation}\label{decirrmfd}
\mathrm{T}^*\mathrm{M}\otimes V_\lambda M\cong \mathrm{TM}\otimes V_\lambda M\cong \underset{\varepsilon\subset\lambda}{\oplus}V_{\lambda+\varepsilon} M
\end{equation}
and the corresponding projections are also denoted by $\Pi_\varepsilon$.

A metric connection $\nabla$ on $\SO_g M$ (or $\Spin_g M$) induces on any associated vector bundle $V_\lambda M$ a covariant derivative, denoted also by $\nabla:\Gamma(V_\lambda M)\to\Gamma(\mathrm{T}^*\mathrm{M}\otimes V_\lambda M)$. The generalized gradients are then built-up by projecting $\nabla$  onto the irreducible subbundles $V_{\lambda+\varepsilon} M$ given by the splitting \eqref{decirrmfd}. 

\begin{Definition}\label{defingg}
Let $(M,g)$ be a Riemannian manifold, $\nabla$ a metric connection and $V_\lambda M$ the vector bundle associated to the irreducible $\SO(n)$ (or $\Spin(n)$)-representation of highest \mbox{weight $\lambda$}. For each relevant weight $\varepsilon$ of $\lambda$, \emph{i.e.} for each irreducible component in the decomposition of $\mathrm{T}^*\mathrm{M}\otimes V_\lambda M$, there is a \emph{generalized gradient} $P_\varepsilon$ defined by the composition:
\begin{equation}\label{defgg}
\Gamma(V_\lambda M) \xrightarrow{\nabla} \Gamma(\mathrm{T^*M}\otimes V_\lambda M) \xrightarrow{\Pi_\varepsilon} \Gamma(V_{\lambda+\varepsilon} M). 
\end{equation}
\end{Definition}

Generalized gradients may be thus defined by any metric connection and the classical case is when $\nabla$ is the Levi-Civita connection. Those defined by the Levi-Civita connection play an important role since they are conformal invariant (see \cite{mpcinv}). The examples given in the sequel are of this type. 

\begin{Example}[Generalized Gradients on Differential Forms]\label{pforms}
We consider the bundle of $p$-forms, $\Lambda^p M$, on a Riemannian manifold $(M^n,g)$ and assume for simplicity that $n=2m+1$ and $p\leq m-1$. The highest weight of the representation is $\lambda_p=(1,\dots,1,0,\dots,0)$ and, by the selection rule in Lemma~\ref{selectrule}, it follows that there are three relevant weights for $\lambda_p$, namely $-\varepsilon_p$, $\varepsilon_{p+1}$ and $\varepsilon_1$. The tensor product then decomposes as follows:
\[\mathrm{TM} \otimes\Lambda^p M\cong\Lambda^{p-1}M\oplus \Lambda^{p+1} M\oplus \Lambda^{p,1} M,\]
where the last irreducible component is the Cartan summand (whose highest weight is equal to the sum of the highest weights of the factors of the tensor product). The generalized gradients in this case are, up to a constant factor, the following:  the codifferential, $\delta$, the exterior derivative, $d$, and respectively the so-called \emph{twistor operator}, $T$.
\end{Example}

\begin{Example}[Dirac and Twistor Operator]\label{exsp}
The spinor representation \mbox{$\rho_n\!:\! \Spin(n) \to\! \Aut(\Sigma_n)$}, with $n$ odd, is irreducible of highest weight $(\frac{1}{2}, \dots, \frac{1}{2})$ and accordingly, on a spin manifold, the tensor product bundle splits into two irreducible subbundles as follows:
\[\mathrm{TM}\otimes \sigm\cong\sigm\oplus \ker(c),\]
where $c:\mathrm{TM}\times \sigm\to\sigm$ denotes the \emph{Clifford multiplication} of a vector field with a spinor.
There are thus two generalized gradients: the \emph{Dirac operator} $D$, which is locally explicitly given by the formula: $D\varphi=\sum_{i=1}^{n}e_i\cdot\nabla_{e_i}\varphi, \quad \text{for all }\varphi\in\Gamma(\sigm)$ (where $\{e_i\}_{i=\overline{1,n}}$ is a local orthonormal basis and the middle dot is a simplified notation for the Clifford multiplication) and the \emph{twistor (Penrose) operator} $T$: $T_X\varphi=\nabla_X\varphi+\frac{1}{n}X\cdot D\varphi$.\\
For $n$ even, the spinor representation splits with respect to the action of the volume element into two irreducible subrepresentations, $\Sigma_n=\Sigma_n^{+}\oplus\Sigma_n^{-}$, of highest weights $(\frac{1}{2}, \dots, \frac{1}{2},\pm\frac{1}{2})$, whose elements are usually called \emph{positive}, respectively \emph{negative half-spinors}. Accordingly, there is a splitting of the spinor bundle: $\sigm=\mathrm{\Sigma^+ M} \oplus \mathrm{\Sigma^- M}$ and the decomposition of the tensor product is then given by: $\mathrm{T}^*\mathrm{M}\otimes \mathrm{\Sigma^{\pm} M}=\mathrm{\Sigma^{\mp} M}\oplus \ker(c)$.
Again the projections onto the first summand correspond to the Dirac operator and onto $\ker(c)$ to the twistor operator.
\end{Example}

\begin{Example}[Rarita-Schwinger Operator]\label{explrs}
Let $n\geq 3$ be odd and consider the so-called {\it twistor bundle}, which is the target bundle of the twistor operator acting on spinors, denoted by $\mathrm{\Sigma_{3/2} M}$. This is the vector bundle associated to the irreducible $\Spin(n)$-representation of highest weight $(\frac{3}{2},\frac{1}{2}, \cdot,\frac{1}{2})$. If $n\geq 5$, it follows from the selection rule in Lemma~\ref{selectrule} that there are four relevant weights: $0$, $-\varepsilon_1$, $+\varepsilon_1$, $+\varepsilon_2$ and the corresponding four gradient targets are: $\mathrm{\Sigma_{3/2} M}$ itself, the spinor bundle $\sigm$,  the associated vector bundles to the irreducible representations of highest weights $(\frac{5}{2}, \frac{1}{2},\dots,\frac{1}{2})$, respectively $(\frac{3}{2},\frac{3}{2},\frac{1}{2},\dots,\frac{1}{2})$. If $n=3$ the last of these targets is missing. The generalized gradient corresponding to the relevant weight $\varepsilon=0$ is denoted by $D_{3/2}:=P^{(3/2,1/2, \dots, 1/2)}_{0}$. This operator is well-known especially in the physics literature and is called the \emph{Rarita-Schwinger operator}.\\
If $n$ is even, $n=2m$, then the two bundles defined by the Cartan summand in $\mathrm{T}^*\mathrm{M}\otimes \mathrm{\Sigma^{\pm} M}$ have highest weights $(\frac{3}{2},\frac{1}{2}, \dots, \frac{1}{2}, \pm\frac{1}{2})$ and the corresponding Rarita-Schwinger operators are the generalized gradients denoted by: $D_{3/2}^{\pm}=P^{(3/2,1/2, \dots, 1/1, \pm 1/2)}_{\mp \varepsilon_m}:\Gamma(\mathrm{\Sigma_{3/2}^{\pm}M})\to\Gamma(\mathrm{\Sigma_{3/2}^{\mp}M})$.
\end{Example}

Essentially the same construction as above may be used to define generalized gradients associated to a $G$-structure. For a study of these $G$-generalized gradients, where $G$ is one of the subgroups of $\SO(n)$ from Berger's list of holonomy groups, we refer the reader \emph{e.g.} to \cite{mpth}.

\section{Branson's Classification of Elliptic Generalized Gradients}\label{brclass}

In this section we present Branson's classification result of \mbox{second} order elliptic operators that naturally arise from generalized gradients. We begin by briefly explaining the notions of ellipticity needed in the sequel and analyze the main properties of these operators given as linear combinations of generalized gradients composed with their formal adjoints. In Theorems~\ref{brclassodd} and~\ref{brclassev} we then state the classification result of Branson, \cite{br1}, of all operators of this type which are elliptic. In particular, this classification shows that ellipticity is attained by assembling surprisingly few generalized gradients. 

If $E$ and $F$ are smooth vector bundles over the manifold $M$ and \mbox{$P\!:\!\Gamma(E)\!\to\!\Gamma(F)$} is a linear differential operator of order $k$, then at every point $x\in M$ and for every $\xi\in T_x^*M$ there is associated an algebraic object, the \emph{principal symbol} $\sigma_\xi(P;x)$, defined invariantly as follows. Let $E_x$ and $F_x$ be the fibers of $E$ and $F$ at $x\in M$, let $u\in\Gamma(E)$ with $u(x)=z$ and $\varphi\in C^\infty(M)$ such that $\varphi(x)=0$, $d\varphi(x)=\xi$, then $\sigma_\xi(P;x):E_x\to F_x$ is the following endomorphism
\begin{equation}\label{defprsymb}
\sigma_\xi(P;x)z=\frac{i^k}{k!}P(\varphi^ku)|_x.
\end{equation}

\begin{Example}\label{prsymb}
On a vector bundle $E$ over a manifold $M$, for any sections $\varphi\in C^\infty(M)$, $u\in\Gamma(E)$, a connection \mbox{$\nabla:\Gamma(E)\to\Gamma(\mathrm{T^*M}\otimes E)$} satisfies $\nabla(\varphi u)=d\varphi\otimes u+\varphi\nabla u$. Consequently, the principal symbol of $\nabla$ is given by:  $\sigma_{\xi}(\nabla)u=i\xi\otimes u$. Alternatively, in the convention without the coefficient $i^k$, the principal symbol is just the identity: $\sigma=\id :\Gamma(\mathrm{T^*M}\otimes E)\to \Gamma(\mathrm{T^*M}\otimes E)$.\\
It follows that the principal symbol of any generalized gradient, which is defined by \eqref{defgg} as a projection of a metric connection onto an irreducible subbundle: $P_\varepsilon:=\Pi_\varepsilon\circ\nabla$, is given exactly by the projection $\Pi_\varepsilon$ defining it, $\sigma_{P_\varepsilon}=\Pi_\varepsilon:\Gamma(\mathrm{T^*M}\otimes V_\lambda M)\to \Gamma(V_{\lambda+\varepsilon} M)$.
\end{Example}

\begin{Definition}\label{defell}
A linear differential operator $P:\Gamma(E)\to \Gamma(F)$ is \emph{elliptic at a point $x\in M$} if its principal symbol $\sigma_\xi(P;x)$ is an isomorphism for every real section $\xi\in \mathrm{T}^*_x\mathrm{M}\setminus\{0\}$. $P$ is \emph{elliptic} if it is elliptic at all points $x\in M$.
\end{Definition}

\begin{Example}
Classical examples of elliptic operators are $\bar\partial$, the Cauchy-Riemann operator acting on complex-valued functions (or more general on forms of type $(0,q)$ on a complex manifold), which is a first order operator and the Laplacian $\Delta$ acting on $p$-forms, which is of second order. These are special cases of elliptic operators obtained from generalized gradients.
\end{Example}

Ellipticity is an algebraic property of a differential operator which implies analytic consequences. The theory of linear elliptic operators is very important and highly-developed. However, in our context we consider a special case of elliptic operators and, as we shall see, the problem may be, without loss of generality, reduced to analyzing first order differential operators.  Obviously, from Definition~\ref{defell}, a necessary condition for the existence of an elliptic operator between two vector bundles is that they have the same rank. So that, in order to talk about the ellipticity of generalized gradients acting between irreducible vector bundles of (usually) different ranks, we need to consider the following notion of ellipticity.

\begin{Definition}\label{injell}
A linear differential operator $P:\Gamma(E)\to \Gamma(F)$ is \emph{underdetermined elliptic at a point $x\in M$} if its symbol $\sigma_\xi(P;x)$ is surjective for every real section $\xi\in \mathrm{T}^*_x\mathrm{M}\setminus\{0\}$. $P$ is \emph{overdetermined} \emph{elliptic at a point $x\in M$} if $\sigma_\xi(P;x)$ is injective for every real section $\xi\in \mathrm{T}^*_x\mathrm{M}\setminus\{0\}$. $P$ is called \emph{(injectively) strongly elliptic} if $\sigma_\xi(P;x)$ is injective for every complex cotangent vector $\xi\in (\mathrm{T}^*_x\mathrm{M})^{\mathbb{C}}\setminus\{0\}$. 
\end{Definition}

\begin{Remark}
Since the principal symbol of a generalized gradient $P_\varepsilon$ is given by the projection $\Pi_\varepsilon$ defining it, the above notion of ellipticity may be easily rephrased in terms of this projection as follows: $P_\varepsilon$ is underdetermined (respectively overdetermined) elliptic if and only if the map $\Pi_\varepsilon\circ (\xi\otimes\cdot):V_\lambda\to V_{\lambda+\varepsilon}$ is surjective (respectively injective), for each nonzero section $\xi\in\Gamma(\mathrm{T}^*_x\mathrm{M})$. Thus, the generalized gradient $P_\varepsilon$ is (strongly) injectively elliptic if and only if  $\Pi_\varepsilon$ is non-vanishing on each decomposable element, \emph{i.e.}
\[\Pi_\varepsilon(\xi\otimes v)=0 \Rightarrow \xi\otimes v=0,\]
where $\xi\in\Gamma(\mathrm{T^*M})$ (respectively $\xi\in\Gamma((\mathrm{T^*M})^{\mathbb{C}})$) and $v\in\Gamma(V_\lambda M)$.
\end{Remark}

\begin{Remark}
A strongly elliptic operator is obviously elliptic. The converse is not true and a counterexample is provided by the Dirac operator $D$ on a spin manifold, whose principal symbol is given by the Clifford multiplication: $\sigma_\xi(D)(\varphi)=\xi\cdot\varphi$.
\end{Remark}

The general setting considered in the sequel is the following: $(M,g)$ is a Riemannian (spin) manifold, $\lambda$ is a dominant weight of $\so(n)$ and $V_\lambda M$ is the associated vector bundle to the irreducible representation of highest weight $\lambda$. For any subset $I$ of the set of relevant weights of $\lambda$, which is completely determined by the selection rule in Lemma~\ref{selectrule}, denote by $D_I$ the following differential operator:
\begin{equation}\label{opdj}
D_I:=\sum_{\varepsilon\in I}P^*_\varepsilon P^{\phantom{*}}_\varepsilon,
\end{equation}
where $P_\varepsilon:=\Pi_\varepsilon\circ\nabla$ is the generalized gradient. It is then natural to ask:

\begin{Question}\label{questell}
Given $\lambda$, for which subsets $I$ is the operator $D_I$ elliptic?
\end{Question}

The complete answer to this question was given by Branson, \cite{br1}. In this section we restate in this notation his result and in \S ~\ref{newpf} we give a new proof of it.

First we notice that Question~\ref{questell} regarding second order differential operators may be reduced to first order ones. If we denote by $P_I$ the following first order operator:
\begin{equation}\label{oppj}
P_I:=\sum_{\varepsilon\in I} P_\varepsilon,
\end{equation}
then $D_I=P_I^*P_I^{\phantom{*}}$ and the following equivalence holds:

\begin{Lemma}\label{equivell}
The operator $D_I$ is elliptic if and only if $P_I$ is \mbox{injectively elliptic}.
\end{Lemma}

This equivalence is a direct consequence of the behavior of principal symbols, namely that: $\sigma_{P_I^*P_I}=(\sigma_{P_I})^*\circ\sigma_{{P_I}^{\phantom{*}}}$, where $(\sigma_{P_I})^*$ is the Hermitian adjoint of $\sigma_{P_I}$. In the sequel we shall shortly say that $P_I$ is elliptic instead of injectively elliptic.

It follows that $D_I$ is elliptic if and only if the projection
\begin{equation}\label{prj}
\Pi_I:=\sum_{\varepsilon\in I}\Pi_\varepsilon: T\otimes V_\lambda \to \underset{\varepsilon\in I}{\oplus} V_{\lambda+\varepsilon}
\end{equation}
is injective when restricted to the set of decomposable elements in $T\otimes V_\lambda$. Thus, the ellipticity of the operators $D_I$ is reduced to a question about the representation theory of $\so(n)$, without reference to any particular manifold. This remark is the starting point in the original proof given by Branson for the classification of elliptic Stein-Weiss operators.

The fact that each projection in \eqref{prj} is onto a different direct summand has the following straightforward, but important consequences:

(1)  If instead of the operators $P_I$ given by \eqref{prj}, we consider, more general\-ly, operators of the form $\sum_{\varepsilon\in I}a_\varepsilon P_\varepsilon$ with nonzero coefficients, then such an operator is elliptic if and only if $P_I$ is. Thus, the ellipticity only depends on the subset $I$ and not on the coefficients, unlike in the case of \wb formulas, where these coefficients play a very important role.

(2) If $I_1\subset I_2$ and $P_{I_1}$ is elliptic, then also $P_{I_2}$ is elliptic. Hence the interesting operators are the \emph{minimal elliptic} operators $P_I$, \emph{i.e.} such that there is no proper subset of $I$ still defining an elliptic operator. It is this set of minimal elliptic operators that was determined by Branson. In a certain sense, the bigger the set $I$ is, the greater is the probability for $P_I$ to be elliptic. For instance, if $I$ is the whole set of weights of the standard representation, then the operator is the \emph{rough Laplacian} $\nabla^*\nabla$, which is, of course, elliptic.

The following partial results concerning Question~\ref{questell} have been shown prior to Branson's classification. Recall that the \emph{Cartan summand} of two irreducible \mbox{representations} $\lambda$ and $\mu$ is the subrepresentation of $\lambda\otimes\mu$ of highest weight $\lambda+\mu$.

\begin{Proposition}[Kalina, Pierzchalski and Walczak, \cite{kpw}]\label{topgrad}
For any irreducible representation $\lambda$, the projection onto the Cartan summand of $\tau\otimes\lambda$ defines a strongly elliptic first order differential operator, also called \emph{top gradient}, and this is the only generalized gradient with this property.
\end{Proposition}

In fact, in \cite{kpw} is proven a  more general version of Proposition~\ref{topgrad}, for the irreducible decomposition of the tensor product of two irreducible representations of any compact semisimple Lie group: among all operators defined by projections from a tensor product to its irreducible subbundles, only the one given by the projection onto the Cartan summand is strongly elliptic. For example, from Proposition~\ref{topgrad}, it follows that the twistor operator $T$ acting on $p$-forms (see Example~\ref{pforms}) is strongly elliptic.

\begin{Proposition}[Stein and Weiss, \cite{sw}]\label{resttop}
For any irreducible re\-presentation $\lambda$, the projection onto the complement of the Cartan summand, \emph{i.e.} when the set $I$ is equal to the whole set of relevant weights except for the highest weight of $\tau$, defines an elliptic operator $P_I$. 
\end{Proposition}

\begin{Example}
In Example~\ref{pforms} the complement of the Cartan projection defines the operator $P=d+\delta$, which, by Proposition~\ref{resttop}, is (injectively) elliptic and, by the above construction, just gives rise to the \emph{Laplacian} acting on $p$-forms: $\Delta=d\delta+\delta d=(d+\delta)^*(d+\delta)$.
\end{Example}

Branson's classification essentially says that the Laplacian is not a special case, but the generalized gradients usually break up into pairs or singletons which are elliptic. Before stating it, we give a graphical interpretation of the classical selection rule in Lemma~\ref{selectrule} for the special orthogonal group, which can be found in \cite{usgw2} and simplifies the task of finding the relevant weights. In our context the graphical interpretation is helpful to better visualize the classification of elliptic operators, which turns out to be strongly related to the selection rule.

First we consider the case of $\SO(n)$ or $\Spin(n)$ with $n$ odd, $n=2m+1$. We use the same notation as above: $\{\pm\e_1,\dots,\pm\e_m,0\}$ are the weights of the defining complex representation $\tau$ of $\SO(n)$ and $V_\l$ is the irreducible representation of highest weight $\l=(\l_1\geq\cdots\geq\l_{m-1}\geq \l_m\geq 0)\in \mathbb{Z}^m\cup (\frac{1}{2}+\mathbb{Z})^m$, where $\l_i$ are the coordinates of $\l$ with respect to the orthonormal basis $\{\e_1,\dots,\e_m\}$. 

\begin{figure}[h]\caption{\label{figsro} Selection Rule for $\SO(2m+1)$ or $\Spin(2m+1)$}
\begin{center}\begin{picture}(430,150)(0,-20)
  \linethickness{.8pt}
  \rectangle(  0, 90)[ 45, 20]{Gray}
  \recttext (  0, 90)[ 45, 20]{$\e_1$}
  \linethickness{.2pt}
  \rectangle( 45, 70)[ 45, 40]{White}
  \recttext ( 45,110)[ 45, 20]{$\l_1\!>\! \l_2$}
  \recttext ( 45, 90)[ 45, 20]{$-\e_1$}
  \recttext ( 45, 70)[ 45, 20]{$\e_2$}
  \rectangle(90, 50)[ 45, 40]{White}
  \recttext (90, 90)[ 45, 20]{$\l_2\!>\! \l_3$}
  \recttext (90, 70)[ 45, 20]{$-\e_2$}
  \recttext (90, 50)[ 45, 20]{$\e_3$}
  \multiput (145, 60)(4,-1){10}{$\cdot$}
  \rectangle(205, 20)[ 45, 40]{White}
  \recttext (205, 60)[ 45, 20]{$\l_{m-2}\!>\! \l_{m-1}$}
  \recttext (205, 40)[ 45, 20]{$-\e_{m-2}$}
  \recttext (205, 20)[ 45, 20]{$\e_{m-1}$}
  \rectangle(250,  0)[ 45, 40]{White}
  \recttext (250, 40)[ 45, 20]{\quad $\l_{m-1}\!>\! \l_{m}$}
  \recttext (250, 20)[ 45, 20]{$-\e_{m-1}$}
  \recttext (250,  0)[ 45, 20]{$\e_m$}
  \rectangle(295,-20)[ 45, 40]{White}
  \recttext (295, 20)[ 45, 20]{$\l_m\!>\!0$}
  \recttext (295, 00)[ 45, 20]{$-\e_m$}
  \recttext (295,-20)[ 45, 20]{$0$}
  \rectangle(297,  0)[ 47, 18]{White}
  \recttext (345,  0)[ 45, 20]{$\l_m\!\geq\! 1$}
 \end{picture}\end{center}
\end{figure}

The decision criterion given in Lemma~\ref{selectrule} for the decomposition of the tensor product $\tau\otimes\lambda$ can be read in Diagram~\ref{figsro} featuring the weights of $\tau$ and labeled boxes. A weight $\varepsilon$ is relevant for an irreducible representation $\lambda$ if and only if the coordinates $\l_1,\dots,\l_m$ of $\lambda$ satisfy all the inequalities labeling the boxes containing $\varepsilon$.

\begin{Theorem}[Branson, \cite{br1}]\label{brclassodd}
Let $(M,g)$ be a Riemannian (spin) manifold of odd real dimension $n=2m+1$ and $V_\lambda M$ the associated vector bundle to an irreducible $\SO(n)$- (or $\Spin(n)$)-representation of highest weight $\lambda$. For any subset $I$ of the set of relevant weights of $\lambda$, the corresponding operator $P_I=\sum_{\varepsilon\in I}\Pi_{\varepsilon}\circ\nabla$ is a minimal elliptic operator if and only if $I$ is one of the following sets:
\begin{enumerate}
\item $\{\e_1\}$ (strongly elliptic),
\item $\{0\}$,  if $\l$ is properly half-integral,
\item $\{-\e_i, \e_{i+1}\}$, for $i=1,\dots,m-1$,
\item $\{-\e_m, 0\}$, if $\l$ is integral.
\end{enumerate}
\end{Theorem}

For the case $n$ even, $n=2m$, there is a similar graphical interpretation of the relevant weights of an irreducible representation $\l=(\l_1\geq\cdots\geq\l_{m-1}\geq |{\l_m}|)\in \mathbb{Z}^m\cup (\frac{1}{2}+\mathbb{Z})^m$, where again $\l_i$ are the coordinates of $\l$ with respect to the orthonormal basis $\{\e_1,\dots,\e_m\}$. 

\begin{figure}[h]\caption{\label{figsre} Selection Rule for $\SO(2m)$ or $\Spin(2m)$}
\begin{center}\begin{picture}(430,150)(0,-20)
  \linethickness{.8pt}
  \rectangle(  0, 90)[ 45, 20]{Gray}
  \recttext (  0, 90)[ 45, 20]{$\e_1$}
  \linethickness{.2pt}
  \rectangle( 45, 70)[ 45, 40]{White}
  \recttext ( 45,110)[ 45, 20]{$\l_1\!>\! \l_2$}
  \recttext ( 45, 90)[ 45, 20]{$-\e_1$}
  \recttext ( 45, 70)[ 45, 20]{$\e_2$}
  \rectangle(90, 50)[ 45, 40]{White}
  \recttext (90, 90)[ 45, 40]{$\l_2\!>\!\l_3$}
  \recttext (90, 70)[ 45, 20]{$-\e_2$}
  \recttext (90, 50)[ 45, 20]{$\e_3$}
  \multiput (145, 60)(4,-1){10}{$\cdot$}
  \rectangle(195, 20)[ 45, 40]{White}
  \recttext (195, 60)[ 45, 20]{$\l_{m-2}\!>\! \l_{m-1}$}
  \recttext (195, 40)[ 45, 20]{$-\e_{m-2}$}
  \recttext (195, 20)[ 45, 20]{$\e_{m-1}$}
  \rectangle(240,  0)[ 45, 40]{White}
  \recttext (242, 40)[ 45, 20]{\quad $\l_{m-1}\!>\! \l_{m}$}
  \recttext (240, 20)[ 45, 20]{$-\e_{m-1}$}
  \recttext (240,  0)[ 45, 20]{$\e_m$}
  \rectangle(242, 20)[ 82, 18]{White}
  \recttext (337, 20)[ 40, 20]{$\l_{m-1}\!>\!-\l_m$}
  \recttext (283, 20)[ 45, 20]{$-\e_m$}
 \end{picture}\end{center}
\end{figure}

\begin{Theorem}[Branson, \cite{br1}] \label{brclassev}
Considering the same assumptions as in Theorem~\ref{brclassodd}, but now on a Riemannian (spin) manifold of even real dimension $n=2m$, the operator $P_I=\sum_{\varepsilon\in I} \Pi_{\varepsilon}\circ\nabla$ is minimal elliptic if and only if $I$ is one of the following sets:
\begin{enumerate}
\item $\{\e_1\}$ (strongly elliptic),
\item $\{-\e_m\}$, if $\l_m>0$, 
\item $\{\e_m\}$, if $\l_m<0$, 
\item $\{-\e_i, \e_{i+1}\}$, for $i=1,\dots,m-2$,
\item $\{-\e_{m-1}, \e_m\}$, if $\l_m\geq 0$, 
\item $\{-\e_{m-1}, -\e_m\}$, if $\l_m\leq 0$.
\end{enumerate}
\end{Theorem}

Theorems~\ref{brclassodd} and~\ref{brclassev} give a complete answer to the Question~\ref{questell}, by restating the results in terms of the second order differential operators $D_I=P_I^*P_I^{\phantom{*}}$. Note that in the list of minimal elliptic operators no operator $P_\varepsilon^*P_\varepsilon^{\phantom{*}}$ appears twice, except for $P_{-\e_{m-1}}^*P_{-\e_{m-1}}^{\phantom{*}}$ in the case when $n$ is even and $\lambda_m=0\neq\lambda_{m-1}$. The list is also exhaustive, \emph{i.e.} each $P_\varepsilon^*P_\varepsilon^{\phantom{*}}$ occurs, except for $n$ odd and $\lambda$ properly half-integral, when $P_{-\e_{m}}^*P_{-\e_{m}}^{\phantom{*}}$ does not occur in the list. Thus, apart from these exceptions, the subsets $I$ defining the minimal elliptic operators form a partition of the set of weights of the standard representation $\tau$.

\begin{Remark}
A priori it is not clear that the ellipticity of the operator $D_I:V_\lambda M\to V_\lambda M$ defined by a certain subset $I$ is independent of the given highest weight $\lambda$ (of course here are considered only the highest weights for which all the elements in $I$ are relevant weights). This follows from Theorems~\ref{brclassodd} and~\ref{brclassev} and no other direct way of proving it is known. 
\end{Remark}

Our aim is to give a new proof of Theorems~\ref{brclassodd} and~\ref{brclassev} in \S ~\ref{newpf}, which seems to be better suited as a starting point for an analogous classification of elliptic operators defined by $G$-generalized gradients for other subgroups $G$ of $\SO(n)$. For this reason we only mention here for the proof of these theorems, that the arguments used by Branson involve tools and techniques of harmonic analysis, explicit computations of the spectra of generalized gradients on the sphere and a strong irreducibility property of principal series representations of the group $\Spin_0(n+1,1)$. For details we refer the reader to Branson, \cite{br1} (see also \cite{br}, \cite{br99}). 

\begin{Remark}
The $\GL(n)$- and $\mathrm{O}(n)$-generalized gradients which are elliptic have been studied by Kalina, \O rsted, Pierzchalski, Walczak and Zhang, \cite{kopwz}, in an elementary way using Young diagrams. However, they find only the top gradient and miss the other elliptic \mbox{generalized} gradients in the list of Branson, because they restrict to a certain special class, the so-called ``up gradients", as pointed out in \cite{br1}.
\end{Remark}

\section{A New Proof of the Classification}\label{newpf}

The aim of this section is to give a new proof of Branson's classification of minimal elliptic (sums of) generalized gradients, \cite{br1}, stated here in our notation in Theorems~\ref{brclassodd} and~\ref{brclassev}. The tools we use are, on the one hand, the relationship between elliptic operators and refined Kato inequalities and, on the other hand, the branching rules for the special orthogonal group. In a first step we extend to all (not necessarily elliptic) generalized gradients the computation of the Kato constant provided by Calderbank, Gauduchon and Herzlich, \cite{cgh}. The main idea is that the argument in \cite{cgh} may be, in a certain sense, reversed: while in \cite{cgh} the purpose is to establish for each natural elliptic ope\-rator an explicit formula of its optimal Kato constant, assuming known the list of Branson of minimal elliptic operators, our goal is to analyze to which extend the computations of the Kato constants rely on this assumption of ellipticity and how Branson's list could be recovered. Our proof does not cover a special case, which is explained in Remark~\ref{spcase}. However, being based mainly on representation theoretical arguments, this proof suggests that it should carry over to other subgroups $G$ of $\SO(n)$, in order to provide the classification of natural elliptic operators constructed from $G$-generalized gradients, as pointed out in Remark~\ref{othersubgps}.

The new proof of the classification result in Theorems~\ref{brclassodd} and~\ref{brclassev} will follow from Propositions~\ref{corcgh} and \ref{cormaxne} and Remark~\ref{listminell}.

\subsection{Elliptic Operators and Refined Kato Inequalities}\label{ellkato}

We first briefly recall what Kato inequalities are and show how refined Kato inequalities are related to the ellipticity of differential operators. 

Kato inequalities are estimates in Riemannian geometry, which have proven to be a powerful technique for linking vector-valued and scalar-valued problems in analysis on manifolds. The classical Kato inequality may be stated as follows. For any section $\varphi$ of a Riemannian or Hermitian vector bundle $E$ endowed with a metric connection $\nabla$ over a Riemannian manifold $(M,g)$, at any point where $\varphi$ does not vanish, the following inequality holds:
\begin{equation}\label{katoin}
|d|\varphi||\leq|\nabla \varphi|.
\end{equation}

This estimate is a direct consequence of the Schwarz inequality applied to the identity $d(|\varphi|^2)=2\ip{\nabla\varphi,\varphi}$, which is given by the fact that the connection is metric. Thus, equality is attained at a point $x\in M$ if and only if $\nabla \varphi$ is a multiple of $\varphi$ at $x$, \emph{i.e.} if there exists a $1$-form $\alpha$ such that
\begin{equation}\label{katoeq}
\nabla\varphi=\alpha\otimes\varphi.
\end{equation}
Whenever \eqref{katoeq} has no solutions in the corresponding geometric setting, there exist refined Kato inequalities, which are of the form 
\begin{equation}\label{refkatoin}
|d|\varphi||\leq k|\nabla \varphi|,
\end{equation}
with a constant $k<1$. For example, such estimates occur in Yau's proof of the Calabi conjecture or in the Bernstein problem for minimal hypersurfaces in $\mathbb{R}^n$. It turns out that the knowledge of the best constant $k$ plays a key role in such proofs. For a survey of these techniques see the introductory part in \cite{cgh} and the references therein.

The principle underlying the existence of refined Kato inequalities was first remarked by J.-P.~Bourguignon, \cite{jpb}. He pointed out that in all geometric settings where refined Kato inequalities occured, the sections under consideration are solutions of a natural linear first order injectively elliptic system and that in such a situation the equality case in \eqref{katoin} cannot be achieved, except at points where $\nabla\varphi=0$. To see this, suppose that equality is attained at a point by a solution $\varphi$ of such an elliptic system. At that point, $\nabla\varphi=\alpha\otimes\varphi$, for some $1$-form $\alpha$. A natural first order operator is one of the form $P_I$ given by \eqref{oppj}. Hence $0=\Pi_I(\nabla\varphi)=\Pi_I(\alpha\otimes \varphi)$ and, by the ellipticity of $P_I$, it follows that $\alpha\otimes \varphi=0$, so $\nabla\varphi=0$. It thus turns out that there is a strong relationship between the ellipticity of the operators $P_I$ and the existence of refined Kato inequalities for sections in their kernel (see Lemma~\ref{equivellkato}).

Calderbank, Gauduchon and Herzlich, \cite{cgh}, proved that there \mbox{exists} indeed a refined Kato inequality for sections in the kernel of any natural first order injectively elliptic operator $P_I$, which acts on sections of a vector bundle associated to an irreducible $\SO(n)$ or $\Spin(n)$-representation. They computed the optimal Kato constant $k_I$, which depends only on the choice of the elliptic operator, in terms of representation-theoretical data. More precisely, the formulas for the optimal Kato constants involve only the conformal weights of the generalized gradients, which are explicitly known. In the sequel we show how this computation can be extended to all (not necessarily elliptic) generalized gradients. In order give our argument we first need to briefly review the main steps in \cite{cgh} for the computation of the optimal Kato constant (see also \cite{cgh2}, \cite{herz}).

The general setting is the same as above and $P_I$ is the operator defined by \eqref{oppj} acting on sections of $V_\lambda M$. In \cite{cgh} it is proven that for each injectively elliptic operator $P_I$, there exists an optimal constant $k_I<1$ such that the refined Kato inequality holds:
\begin{equation}\label{katoinp}
|d|\varphi||\leq k_I|\nabla \varphi|,\quad \text{ for all } \varphi\in\ker(P_I),
\end{equation}
and an explicit formula is given for $k_I$, in terms of the translated conformal weights (see Theorem~\ref{thmcgh}). We recall that the conformal weights are the eigenvalues of the so-called conformal weight operator defined as follows:

\begin{Definition}\label{defconfop}
The {\it conformal weight operator} of an $\SO(n)$-representation \mbox{$\lambda: \SO(n) \to \Aut(V)$}, is the symmetric endomorphism defined as follows:
\begin{equation}\label{confop}
B: (\mathbb{R}^n)^*\otimes V \to (\mathbb{R}^n)^*\otimes V, \quad B(\alpha\otimes v)=\sum_{i=1}^{n}e_i^*\otimes d\lambda(e_i\wedge\alpha)v,
\end{equation}
where $\{e_i\}_{\overline{1,n}}$ is an orthonormal basis of $\mathbb{R}^n$ and $\{e_i^*\}_{\overline{1,n}}$ its dual basis. We also denote by $B$ the induced endomorphism on the associated bundle \mbox{$\mathrm{T}^*M\!\otimes \!V_{\lambda} M$}.
\end{Definition}

As pointed out in \cite{cgh}, the computations are simplified if one considers the \emph{translated conformal weight operator}:
\begin{equation}\label{translcw}
\widetilde{B}:(\mathbb{R}^n)^*\otimes V_\lambda\to (\mathbb{R}^n)^*\otimes V_\lambda,\quad \widetilde{B}:=B+\frac{n-1}{2}\mathrm{Id},
\end{equation}
whose eigenvalues are the \emph{translated conformal weights}, which explicitly known (see \emph{e.g.} \cite{feg}):
\begin{equation}\label{confwtrans}
\begin{split}
\widetilde{w}_0(\lambda) &=0,\\
\widetilde{w}_{i,+}(\lambda) &=\lambda_i-i+\frac{n+1}{2} , \text{ for } i=1,\dots, m,\\
\widetilde{w}_{i,-}(\lambda) &=-\lambda_i+i-\frac{n-1}{2}, \text{ for } i=1,\dots, m.
\end{split}
\end{equation}
The key property used in the sequel is that the (translated) conformal weights are strictly ordered, with the exception of the case when $n$ is even, $n=2m$, $\lambda_m=0$ and $\tilde{w}_{m,+}=\tilde{w}_{m,-}$, which is due to the fact that the two corresponding $\SO(n)$-irreducible representations are exchanged by a change of orientation, while their sum is an irreducible $\mathrm{O}(n)$-representation. In this exceptional case these two representations are considered as one summand, so that the conformal weights of distinct projections are always different from each other. The translated conformal weights have the advantage that the virtual weights whose relevance depends on the same condition on $\lambda$, \emph{i.e} that are in the same box in the Diagrams~\ref{figsro} and~\ref{figsre}, sum up to zero if that condition is not fulfilled. For instance, if $\lambda_i=\lambda_{i+1}$, then the weights $-\varepsilon_{i}$ and $\varepsilon_{i+1}$ are not relevant for $\lambda$ and their corresponding conformal weights satisfy: $\widetilde{w}_{i,-}(\lambda)+\widetilde{w}_{i+1,+}(\lambda)=0$. This cancellation property for non-relevant weights is useful for the forthcoming computations.

The strict ordering of the translated conformal weights allows us to rename them (and the corresponding summands in the decomposition of the tensor product $(\mathbb{R}^n)^*\otimes V_\lambda$) and to index them in a decreasing ordering as follows: 
\begin{equation}\label{descod}
(\mathbb{R}^n)^*\otimes V_\lambda=\overset{N}{\underset{i=1}{\oplus}} V_i,
\end{equation}
with
\[\widetilde{w}_1(\lambda)>\widetilde{w}_2(\lambda)>\cdots>\widetilde{w}_N(\lambda),\]
where $N$ is the number of summands in the decomposition, \emph{i.e.} the number of relevant weights for $\lambda$. This reordering of the indices of the conformal weights is then carried over to the corresponding weights of the standard representation and thus, the subsets $I$ defining the operators $P_I$ are subsets of $\{1,\dots,N\}$. 

\begin{Remark}\label{listminell}
Notice that, in the above notation, the weights which are in the same box in Diagram~\ref{figsro} and~\ref{figsre} are pairs of type $\{i,N+2-i\}$ and the list of minimal elliptic operators of the form $P_I$ established by Branson (see Theorems~\ref{brclassodd} and~\ref{brclassev}) is the following:
\begin{enumerate}
\item $P_{\{1\}}$;
\item $P_{\{\ell+1\}}$ if $N=2\ell$ and $\lambda_m\neq 0$;
\item $P_{\{\ell\}}$ if $N=2\ell-1$ and $\lambda$ is properly half-integral;
\item $P_{\{i,N+2-i\}}$ for $i=2,\dots,\ell-1$;
\item $P_{\{\ell,\ell+2\}}$ if $N=2\ell$;
\item $P_{\{\ell,\ell+1\}}$ if $N=2\ell-1$ and $\lambda$ is integral.
\end{enumerate}
In particular, the list of the minimal elliptic operators depends only on the ordering of the conformal weights.
\end{Remark}

The following result reduces the search for refined Kato inequalities to an algebraic problem. The complement of $I$ in $\{1,\dots,N\}$ is denoted by $\widehat{I}$.

\begin{Lemma}(\cite{cgh})\label{katoct}
Let $I$ be a subset of $\{1,\dots,N\}$ and $P_I:=\sum_{i\in I}\Pi_i\circ \nabla$ the corresponding operator. For any section $\varphi$ in the kernel of $P_I$ and at any point where $\varphi$ does not vanish, the following inequality holds:
\begin{equation}\label{katocteq}
|d|\varphi||\leq k_I|\nabla \varphi|,
\end{equation}
where the constant $k_I$, called \emph{Kato constant}, is defined by
\begin{equation}\label{katoct1}
k_I := \underset{|\alpha|=|v|=1} {\sup}|\Pi_{\widehat I}(\alpha\otimes v)|=\sqrt{1-\underset{|\alpha|=|v|=1} {\inf}|\Pi_{I}(\alpha\otimes v)|^2},
\end{equation}
where $\alpha\in(\mathbb{R}^n)^*$ and $v\in V_\lambda$. Furthermore, equality holds at a point if and only if $\nabla\varphi=\Pi_{\widehat I}(\alpha\otimes\varphi)$ for a $1$-form $\alpha$ at that point, such that: 
\[|\Pi_{\widehat{I}}(\alpha\otimes\varphi)|=k_I|\alpha\otimes\varphi|.\]
\end{Lemma}

The proof of Lemma~\ref{katoct} is based, as for the classical Kato inequality, on purely algebraic refined Schwarz inequalities of the form:
\begin{equation}\label{schwarz}
\frac{|\langle\Phi, v\rangle|}{|v|}\leq k|\Phi|,
\end{equation}
where $\Phi\in (\mathbb{R}^n)^*\otimes V_\lambda$ and $v\in V_\lambda$. The inequality \eqref{katocteq} is obtained by lifting \eqref{schwarz} to the associated vector bundles and putting $v=\varphi$ and $\Phi=\nabla\varphi$ for a section $\varphi\in\Gamma(V_\lambda M)$. 

The first step in the minimization process for the computation of the Kato constant $k_I$, given by \eqref{katoct1}, is to express the norm of each projection $\Pi_j$, $j=1,\dots, N$, as an affine function as follows: for $N=2\ell-1$
\begin{equation}\label{projodd}
|\Pi_j(\alpha\otimes v)|^2= \frac{\displaystyle \widetilde{w}_j^{2(\ell-1)}- \sum_{k=2}^{\ell} (-1)^k\widetilde{w}_j^{2(\ell-k)}Q_k}
{\displaystyle\prod_{k\neq j}(\widetilde{w}_j- \widetilde{w}_k)}=:\pi_j(Q_2,\dots,Q_\ell),
\end{equation}
and for $N=2\ell$
\begin{equation}\label{projeven}
|\Pi_j(\alpha\otimes v)|^2= \left(\widetilde{w}_j-\frac{1}{2}\right)\frac{\displaystyle \widetilde{w}_j^{2(\ell-1)}- \sum_{k=2}^{\ell} (-1)^k\widetilde{w}_j^{2(\ell-k)}Q_k}
{\displaystyle\prod_{k\neq j}(\widetilde{w}_j- \widetilde{w}_k)}=:\pi_j(Q_2,\dots,Q_\ell),
\end{equation}
with the variables $Q_k$ given by: $Q_k:=(-1)^{k-1}\ip{ A_{2k-2}(\alpha\otimes v),\alpha\otimes v}$,  $k=2,\dots,\ell$, for $N=2\ell-1$ or $N=2\ell$, 
where $A_k:=\sum_{\ell=0}^k(-1)^\ell\sigma_\ell(\widetilde{w})\widetilde{B}^{k-\ell}$ and $\sigma_i(\widetilde{w})$ is the $i$-th elementary symme\-tric function in the translated conformal weights $\widetilde{w}_1,\dots, \widetilde{w}_N$.

Hence, the problem of estimating $\underset{|\alpha|=|v|=1} {\inf}|\Pi_{I}(\alpha\otimes v)|^2$ (for a subset $I$ corres\-ponding to an elliptic operator) is reduced to minimizing this affine function over the admissible region in the $(\ell-1)$-dimensional affine space. The admissible region consists of the points $Q$ of coordinates $\{Q_k\}_{k=\overline{2,\ell}}$, such that there exist unitary vectors $\alpha\in(\mathbb{R}^n)^*$ and $v\in V_\lambda$ with the property that for each $k=2,\dots, \ell$ the following relation holds: \mbox{$Q_k=(-1)^{k-1}\ip{ A_{2k-2}(\alpha\otimes v),\alpha\otimes v}$}. Thus, the search for Kato constants mainly reduces to linear programming.

The admissible region is contained in a convex in the $Q$-space, since $|\Pi_j(\alpha\otimes v)|^2=\pi_j(Q)$ and each norm is non-negative and smaller than $1$, if $Q$ is an admissible point. More precisely, from \eqref{projodd} it follows that the point $Q=(Q_2,\ldots ,Q_\ell)$ is in the convex region $\mathcal{P}$ in $\R^{\ell-1}$ defined by the following system of linear inequalities for $N$ odd, $N=2\ell-1$:
\begin{equation}\label{oddsys}
\sum_{k=2}^{\ell}(-1)^{j+k} \widetilde{w}_j^{2(\ell-k)}Q_k
\geq(-1)^j \widetilde{w}_j^{2(l-1)}, \quad j=1,\dots,2\ell-1,
\end{equation}
with equality if and only if $|\Pi_j(\alpha\otimes v)|^2=\pi_j(Q)=0$. For $N$ even, $N=2\ell$, the system of linear inequalities is similarly obtained from \eqref{projeven}, taking into account that the sign of the denominator in \eqref{projeven} is $(-1)^{j-1}$:
\begin{equation}\label{evensys}
\begin{cases}
\overset{\ell}{\underset{k=2}{\sum}}(-1)^{j+k} \widetilde{w}_j^{2(\ell-k)}Q_k \geq(-1)^j \widetilde{w}_j^{2(l-1)}, \quad 1\leq j\leq\ell,\\
\overset{\ell}{\underset{k=2}{\sum}}(-1)^{j+k} \widetilde{w}_j^{2(\ell-k)}Q_k \leq(-1)^j \widetilde{w}_j^{2(l-1)}, \quad \ell+1\leq j\leq 2\ell,
\end{cases}
\end{equation}
and equality is attained if and only if $|\Pi_j(\alpha\otimes v)|^2=\pi_j(Q)=0$.

The convex region $\mathcal{P}$ defined by the system \eqref{oddsys}, respectively \eqref{evensys}, is proven in \cite{cgh} to be compact, hence polyhedral. Since the norms are affine in the $Q_k$'s, it then suffices to minimize over the set of vertices.

For a subset $J\subset\{1,\dots,N\}$ with $\ell-1$ elements, the intersection of the corresponding hyperplanes is the point denoted by $Q^J$:
\begin{equation}\label{qj}
\{Q^J\}:=\underset{j\in J}{\cap} \{\pi_j (Q_2,\ldots Q_\ell)=0\},
\end{equation}
whose coordinates are given by the elementary symmetric functions in the squares of the translated conformal weights: $Q^J_k = \sigma_{k-1}\bigl((\widetilde{w}_j^2)_{j\in J}\bigr)$. At the point $Q^J$, the affine functions $\pi_j$, defined by \eqref{projodd} for $N=2\ell-1$, take the values
\begin{equation}\label{oddsoln}
\pi_j(Q^J)=
\frac{\displaystyle\prod_{k\in J}( \widetilde{w}_j^2-\widetilde{w}_k^2)}
{\displaystyle\prod_{k\neq j}( \widetilde{w}_j- \widetilde{w}_k)}
=\frac{\displaystyle\prod_{k\in J, k\neq j}( \widetilde{w}_j+ \widetilde{w}_k)}
{\displaystyle\prod_{k\in \widehat J, k\neq j}( \widetilde{w}_j- \widetilde{w}_k)}
\,\varepsilon_j(J),
\end{equation}
where $\varepsilon_j(J)=0$ if $j\in J$ and $1$ otherwise. Similarly, for $N=2\ell$, the affine functions $\pi_j$ defined by \eqref{projeven} take the values:
\begin{equation}\label{evensoln}
\pi_j(Q^J)=
\left(\widetilde{w}_j-\frac{1}{2}\right)\frac{\displaystyle\prod_{k\in J}( \widetilde{w}_j^2-\widetilde{w}_k^2)}
{\displaystyle\prod_{k\neq j}( \widetilde{w}_j- \widetilde{w}_k)}
=\left(\widetilde{w}_j-\frac{1}{2}\right)\frac{\displaystyle\prod_{k\in J, k\neq j}( \widetilde{w}_j+ \widetilde{w}_k)}
{\displaystyle\prod_{k\in \widehat J, k\neq j}( \widetilde{w}_j- \widetilde{w}_k)}
\,\varepsilon_j(J).
\end{equation}

As there exists a set of minimal elliptic operators, there also exists a set of \emph{maximal non-elliptic} operators, \emph{i.e.} the set of operators $P_I$ which are non-elliptic and $I$ has maximal cardinality. Theorems~\ref{brclassodd} and ~\ref{brclassev} provide us also the set of maximal non-elliptic operators, which are explicitly described as follows.

Let $\NE$ denote the set of subsets of $\{1,\dots, N\}$ whose elements are obtained by choosing exactly one index in each of the sets $\{j,N+2-j\}$ for $2\leq j \leq \ell$, if $N=2\ell-1$ or $N=2\ell$, giving $2^{\ell-1}$ elements:
\begin{equation}\label{neset}
\NE=\{J\subset\{1,\dots,N\}\,|\, |J\cap\{i,N+2-i\}|=1, \text{for } 2\leq i\leq \ell\}.
\end{equation}
The elements of $\NE$ are then precisely the subsets of $\{1,\dots,N\}$ corresponding to the maximal non-elliptic operators, unless $n$ is odd, $N=2\ell-1$ and $\lambda$ is properly half-integral, in which case the subsets containing $\ell$ (which corresponds to the zero weight) are elliptic. This is called the \emph{exceptional case} and is the only one when the Kato constant provided by Theorem~\ref{thmcgh} might not be optimal.

The set $\NE$ can easily be described in the graphical interpretation given by Diagrams~\ref{figsro} and ~\ref{figsre} (with the remark that now the indices are considered according to the convention given by the decreasing ordering of the conformal weights): each element of $\NE$ contains exactly one index from each box containing two weights. For instance, for $n=2m+1$, if $\lambda_m=\frac{1}{2}$, then $-\varepsilon_m$ is not relevant and the zero weight forms one box, so that it is not taken in any subset in $\NE$; if $\lambda_m\geq 1$, then $\{-\varepsilon_m,0\}$ are in the same box and one of them is chosen for each subset in $\NE$. For $n=2m$, if $\lambda_{m-1}>\lambda_m>0$, then $\{-\varepsilon_{m-1},\varepsilon_m\}$ form one box and $\{-\varepsilon_m\}$ is alone in a box; whereas if \mbox{$\lambda_{m-1}>-\lambda_m>0$}, then $\pm\varepsilon_m$ are interchanged (since the ordering of the corresponding conformal weights changes: $w_{m,+}-w_{m,-}=2\lambda_m$), namely $\{-\varepsilon_{m-1},-\varepsilon_m\}$ are in one box and $\{\varepsilon_m\}$ forms itself a box. This is in accordance with the classification of minimal elliptic operators: $\{-\varepsilon_m\}$ and $\{-\varepsilon_{m-1},\varepsilon_m\}$ are elliptic if $\lambda_m>0$, and $\{\varepsilon_m\}$ and $\{-\varepsilon_{m-1},-\varepsilon_m\}$ are elliptic if $\lambda_m<0$. In the special case when the weights $\pm\varepsilon_m$ are relevant and their conformal weights are equal: $w_{m,-}=w_{m,+}$, \emph{i.e.} when $\lambda_{m-1}>\lambda_m=0$, then, in our convention, the corresponding representations are considered as one summand $V_{\lambda-\varepsilon_m}\oplus V_{\lambda+\varepsilon_m}$ and in this case the last box in the Diagram~\ref{figsre} is formed by $\{-\varepsilon_{m-1}, \varepsilon_m\}$, and the corresponding projection to $\varepsilon_m$ is here actually the projection onto $V_{\lambda-\varepsilon_m}\oplus V_{\lambda+\varepsilon_m}$. 

In the sequel we call $\NE$ the set of maximal non-elliptic operators, which is true apart from the exceptional case. Notice that each subset in $\NE$ has exactly $\ell-1$ elements, where $\ell$ gives the parity of $N$, \emph{i.e.} $N=2\ell-1$ or $N=2\ell$. The explicit description of the vertices of the polyhedral region $\mathcal{P}$ in $\R^{\ell-1}$ is the following:

\begin{Proposition}[\cite{cgh}]\label{vertices}
The vertices of the polyhedron $\mathcal{P}$ are exactly the points $Q^J$, defined by \eqref{qj}, with $J\in\NE$, the set of maximal non-elliptic operators. In the exceptional case, when $n$ is odd and $\lambda$ is properly half-integral, only one inclusion holds, namely that the vertices are contained in the set $\NE$.
\end{Proposition}

In order to compute the Kato constant given by \eqref{katoct1} it suffices to minimize or maximize over the set of vertices of the polyhedron $\mathcal{P}$. The identification of these vertices provided by Proposition~\ref{vertices} and the explicit computation of the norms $|\Pi_{j}(\alpha\otimes v)|^2$ at each vertex prove the following:

\begin{Theorem}[Calderbank, Gauduchon and Herzlich, \cite{cgh}]\label{thmcgh}
Let $I$ be a subset of $\{1,\dots,N\}$ corresponding to an injectively elliptic operator $P_I=\sum_{i\in I}\Pi_i\circ \nabla$ acting on sections of $V_\lambda M$. Then a refined Kato inequality holds: \mbox{$|d|\varphi||\leq k_I|\nabla \varphi|$}, for any section $\varphi\in\ker(P_I)$, outside the zero \mbox{set of $\varphi$}.\\
If $N$ is odd, the Kato constant $k_I$ is given by the following expressions:
\begin{equation}\label{Nodd} k_I^2 = \max_{J\in\NE}
\left(\,\sum_{i\in\widehat I\intersect\widehat J} \frac{\prod_{j\in J} ( \widetilde{w}_i
+  \widetilde{w}_j)}{\prod_{j\in\widehat J\setminus\{i\}} ( \widetilde{w}_i -
 \widetilde{w}_j)} \right)
= 1 - \min_{J\in\NE} \left(\,\sum_{i\in I\intersect\widehat J} \frac{\prod_{j\in
J} ( \widetilde{w}_i + \widetilde{w}_j)}{\prod_{j\in \widehat J\setminus\{i\}}
( \widetilde{w}_i - \widetilde{w}_j)} \right).
\end{equation}
If $N$ is even, the Kato constant $k_I$ is similarly given by:
\begin{equation}\label{Neven} 
\begin{split}
k_I^2 &= \max_{J\in\NE}
\left(\,\sum_{i\in\widehat I\intersect\widehat J} \left(\widetilde{w}_i-\frac{1}{2}\right) \frac{\prod_{j\in J} ( \widetilde{w}_i
+  \widetilde{w}_j)}{\prod_{j\in\widehat J\setminus\{i\}} ( \widetilde{w}_i -
 \widetilde{w}_j)} \right)\\
&= 1 - \min_{J\in\NE} \left(\,\sum_{i\in I\intersect\widehat J}\left(\widetilde{w}_i-\frac{1}{2}\right) \frac{\prod_{j\in
J} ( \widetilde{w}_i + \widetilde{w}_j)}{\prod_{j\in \widehat J\setminus\{i\}}
( \widetilde{w}_i - \widetilde{w}_j)} \right).
\end{split}
\end{equation}
These Kato constants are optimal, unless in the exceptional case when $n$ and $N$ are odd, $N=2\ell+1$, $\lambda$ is properly half-integral and the set $J$ achieving the extremum contains $\ell+1$.
\end{Theorem}

\begin{Remark}
A completely different approach to the computation of optimal Kato constants was provided, independently, by Branson, \cite{br2}, whose proof relies on powerful techniques of harmonic analysis. One may say that the method in \cite{cgh} is the local method, relying on algebraic considerations on the conformal weights and a linear programming problem. On the other hand, the method in \cite{br2} is a global one, using the spectral computation on the round sphere in \cite{br1} and a result relating the spectrum of an operator to information on its symbol. The advantage of the local method is that it provides an explicit description of the sections satisfying the equality case of the refined Kato inequality, while the advantage of the global method is that it is sharp also in the exceptional case.
\end{Remark}

The starting point in our new proof is the following straightforward observation:

\begin{Lemma}\label{equivellkato}
Let $k_I$ be the optimal Kato constant for the operator $P_I$, which is given by $k_I=\underset{|\alpha|=|v|=1} {\sup}|\Pi_{\widehat I}(\alpha\otimes v)|$ (see Lemma~\ref{katoct}). Then the following equivalence holds:
\[k_I<1\; \Longleftrightarrow\; P_I \text{ is an elliptic operator.}\]
\end{Lemma}

\begin{proof}
If $|\alpha|=|v|=1$, then \mbox{$1=|\alpha\otimes\varphi|^2=|\Pi_I(\alpha\otimes\varphi)|^2+|\Pi_{\widehat{I}}(\alpha\otimes\varphi)|^2$}, so that $k_I$ is always smaller or equal to $1$. Then, by negation, the equivalence in the statement is the same as the following equivalence:
\[k_I=1\; \Longleftrightarrow\; P_I \text{ is not elliptic},\]
which is a consequence of the definitions: $k_I=1$ if and only if there exist $\alpha$ and $v$ of norm $1$ such that $|\Pi_{\widehat{I}}(\alpha\otimes v)|=1$, which is then the same as $|\Pi_I(\alpha\otimes\varphi)|=0$, or, equivalently, $\alpha\otimes\varphi\in\ker(P_I)$, meaning that $P_I$ is not elliptic.\hfill $\qed$
\end{proof}

Lemma~\ref{equivellkato} implies that the ellipticity of a natural first order differential operator $P_I$ follows from the computation of its optimal Kato constant $k_I$. Thus, as soon as we are able to compute explicitly $k_I$ (without using the ellipticity assumption) or to show that $k_I$ is strictly less than $1$, it follows that the operator $P_I$ is elliptic. In the sequel we show that $k_I$ is strictly bounded from above by $1$ for the operators in Branson's list (\emph{i.e.} in the notation given by the decreasing ordering of the translated conformal weights, for all operators enumerated in Remark~\ref{listminell}), except the case 3., which corresponds to the zero weight.

We use the same notation as above and notice that for the construction of the convex region $\mathcal{P}$, as well as for establishing its compactness, the only ingredient needed is the ordering of the translated conformal weights, which is provided by the explicit formulas \eqref{confwtrans}.
 
The key observation is that the only step in the proof of Theorem~\ref{thmcgh} in \cite{cgh} where the ellipticity of the operators is used, is in the identification of the vertices of the polyhedral region, namely in Proposition~\ref{vertices}. If we now consider the same set $\NE$ introduced in \eqref{neset}, then one inclusion established in \mbox{Proposition~\ref{vertices}} still holds, without any ellipticity assumption on the operators. More precisely, we obtain:
\begin{Lemma}\label{lemvert}
The vertices of the polyhedron $\mathcal{P}$ are given by a subset of $\NE$.
\end{Lemma}

\begin{proof}
Let us denote by $\mathcal{V}$ the set of vertices of the polyhedron $\mathcal{P}$ in $\R^{\ell-1}$, which are characterized as follows:
\[\mathcal{V}=\{Q^J\,|\, |J|=\ell-1, \Pi_j(Q^J)=0, \text{for all } j\in J;\, \Pi_j(Q^J)>0, \text{for all } j\in \widehat{J}\}.\]
Then we have to show the following inclusion: $\mathcal{V}\subset\{Q^J|\, J\in\NE\}$. Or, equivalently, we prove that $J\notin\NE$ implies $Q^J
\notin\mathcal{V}$ (where $J$ is a subset of $\{1, \dots, N\}$ with $\ell-1$ elements, for $N=2\ell$ or $N=2\ell-1$).

Let $J\notin\NE$. In order to show that $Q^J$ is not a vertex of the polyhedron $\mathcal{P}$ it is
enough to find an element $i\in\{1, \dots, N\}$ such that $\pi_i (Q^J)<0$.

For $N$ odd, equation \eqref{oddsoln} implies that for each $i\notin J$, $\Pi_i(Q^J)$ is nonzero and its sign is:
\[\sgn(\pi_i(Q^J))=(-1)^{i-1}\sgn(\prod_{j\in J}( \widetilde{w}_i^2 - \widetilde{w}_j^2 )).\]
There are exactly $\ell-1$ couples of the type $(s,N+2-s)$ and, since $J\notin\NE$ and has $\ell-1$ elements, there exists at least one such couple not contained in $J$. \\
The ordering of the squares of the translated conformal weights, that can be directly checked by the formulas \eqref{confwtrans}, is the following ($N=2\ell-1$): 
\[\widetilde{w}_1^2>\widetilde{w}_{N+1}^2>\widetilde{w}_2^2>\widetilde{w}_{N}^2>\cdots>\widetilde{w}_i^2>\widetilde{w}_{N+2-i}^2>\cdots> \widetilde{w}_{\ell}^2>\widetilde{w}_{N+2-\ell}^2.\]
It then follows that for a couple $(s,N+2-s)$, $\widetilde{w}_{s}^2$ and $\widetilde{w}_{N+2-s}^2$ are adjacent in this ordering, so that the following signs are the same: 
\[\sgn(\prod_{j\in J}( \widetilde{w}_s^2 - \widetilde{w}_j^2 ))=\sgn(\prod_{j\in J}( \widetilde{w}_{N+2-s}^2 - \widetilde{w}_j^2 )).\] 
Since $N$ is odd, $s$ and $N+2-s$ have different parity, showing that $\pi_s (Q^J)$ and $\pi_{N+2-s} (Q^J)$  have opposite signs.

For $N$ even, the only difference is the way the sign chances when passing from $i=s$ to $i=N+2-s$: the parity of $i$ remains the same, but the sign of the factor $\left(\widetilde{w}_j-\frac{1}{2}\right)$ in \eqref{evensoln} changes.\hfill $\qed$
\end{proof}

From the inclusion $\mathcal{V}\subset\NE$ given by Lemma~\ref{lemvert}, the formula \eqref{katoct1} for the Kato constant $k_I$ and the expressions \eqref{projodd} and \eqref{projeven} for the norms of the projections, we obtain the following upper bound:

\begin{Proposition}\label{propuppb}
Let $I$ be a subset of $\{1,\dots,N\}$ and the operator $P_I=\sum_{i\in I}\Pi_i\circ \nabla$ acting on sections of $V_\lambda M$. Then the corresponding Kato constant $k_I$ satisfies the following inequality:
\begin{equation}\label{boundkato}
k^2_I = \underset{Q\in\mathcal{P}} {\max}\left(\sum_{j\in\widehat{I}}\pi_j(Q)\right)= \underset{Q\in\mathcal{V}} {\max}\left(\sum_{j\in\widehat{I}}\pi_j(Q)\right)\leq  \underset{J\in\mathcal{\NE}} {\max}\left(\sum_{j\in\widehat{I}}\pi_j(Q^J)\right)=:c_I.
\end{equation}
Thus, if $c_I<1$ for a subset $I\subset\{1,\dots,N\}$, it follows by Lemma~\ref{equivellkato} that the corresponding operator $P_I$ is elliptic.
\end{Proposition}

We notice that the formulas for the optimal Kato constant in Theorem~\ref{thmcgh} actually compute the values of the upper bound $c_I$, if we do not assume the ellipticity of any operator involved. This straightforward, but important remark provides the main argument in our proof of Branson's classification.

From Theorem~\ref{thmcgh} applied to the special case when the set $I$ has only one element or two elements of the form ${\{i,N+2-i\}}$, one recovers the list of minimal elliptic operators as follows.

\begin{Proposition}\label{corcgh}
The upper bound $c_I$ is strictly smaller than $1$ for any of the following subsets $I$: 
\begin{enumerate}
\item $I=\{1\}$;
\item $I=\{\ell+1\}$ if $N=2\ell$ and $\lambda_m\neq 0$;
\item $I=\{i,N+2-i\}$ for $i=2,\dots,\ell$.
\end{enumerate}
From the above discussion it follows that the corresponding operators $P_I$ are elliptic. 
\end{Proposition}

\begin{proof}
By Theorem~\ref{thmcgh}, the upper bound $c_I$ is given by the following formula, if $N=2\ell-1$:
\begin{equation}\label{Nodd2} 
\begin{split}
c_I &= \underset{J\in\mathcal{\NE}} {\max}\left(\sum_{j\in\widehat{I}}\pi_j(Q^J)\right)= \max_{J\in\NE}
\left(\,\sum_{i\in\widehat I\intersect\widehat J} \frac{\prod_{j\in J} ( \widetilde{w}_i
+  \widetilde{w}_j)}{\prod_{j\in\widehat J\setminus\{i\}} ( \widetilde{w}_i -
 \widetilde{w}_j)} \right)\\
&= 1 - \min_{J\in\NE} \left(\,\sum_{i\in I\intersect\widehat J} \frac{\prod_{j\in
J} ( \widetilde{w}_i + \widetilde{w}_j)}{\prod_{j\in \widehat J\setminus\{i\}}
( \widetilde{w}_i - \widetilde{w}_j)} \right),
\end{split}
\end{equation}
and if $N=2\ell$:
\begin{equation}\label{Neven2} 
\begin{split}
c_I &= \underset{J\in\mathcal{\NE}} {\max}\left(\sum_{j\in\widehat{I}}\pi_j(Q^J)\right)=\max_{J\in\NE}
\left(\,\sum_{i\in\widehat I\intersect\widehat J} \left(\widetilde{w}_i-\frac{1}{2}\right) \frac{\prod_{j\in J} ( \widetilde{w}_i
+  \widetilde{w}_j)}{\prod_{j\in\widehat J\setminus\{i\}} ( \widetilde{w}_i -
 \widetilde{w}_j)} \right)\\
&= 1 - \min_{J\in\NE} \left(\,\sum_{i\in I\intersect\widehat J}\left(\widetilde{w}_i-\frac{1}{2}\right) \frac{\prod_{j\in
J} ( \widetilde{w}_i + \widetilde{w}_j)}{\prod_{j\in \widehat J\setminus\{i\}}
( \widetilde{w}_i - \widetilde{w}_j)} \right).
\end{split}
\end{equation}

The last expressions in \eqref{Nodd2} and \eqref{Neven2} are particularly simple if the set $I$ has just a few elements, as it is in our case. 

1. Substituting $I=\{1\}$ in \eqref{Nodd2} and \eqref{Neven2}, the sums reduce to one element, since $I\cap\widehat{J}=\{1\}$ for any $J\in\NE$, and we get:
\begin{equation}\label{I1odd} 
c_{\{1\}}= 1 - \min_{J\in\NE} \left(\frac{\prod_{j\in
J} ( \widetilde{w}_1 + \widetilde{w}_j)}{\prod_{j\in \widehat J\setminus\{i\}}
( \widetilde{w}_1 - \widetilde{w}_j)} \right), \quad \text{if } N=2\ell-1,
\end{equation}
\begin{equation}\label{I1even} 
c_{\{1\}} = 1 - \min_{J\in\NE} \left(\left(\widetilde{w}_1-\frac{1}{2}\right)\frac{\prod_{j\in
J} ( \widetilde{w}_1 + \widetilde{w}_j)}{\prod_{j\in \widehat J\setminus\{1\}}
( \widetilde{w}_1 - \widetilde{w}_j)} \right), \quad \text{if } N=2\ell,
\end{equation}
which implies that $c_{\{1\}}<1$, because $\widetilde{w}_1$ is the biggest translated conformal weight: $\widetilde{w}^2_1>\widetilde{w}^2_j$, for any $2\leq j\leq N$ and $\widetilde{w}_1=\lambda_1+\frac{n-1}{2}>\frac{1}{2}$ (we assume always $n\geq 2$ and $\lambda_1\neq 0$, otherwise $\lambda$ is just the trivial representation).

2. If the dimension $n$ is odd, $n=2m+1$, the case $N=2\ell$ can only occur if $\lambda_m=\frac{1}{2}$, as can be easily seen in the Diagram~\ref{figsro} which illustrates the selection rule (since in all the other cases the weights come in pairs). In this case, the index $\ell+1$, given by the decreasing ordering of the translated conformal weights, stays for the weight $0$. If $n=2m$ and $N=2\ell$, then from Diagram~\ref{figsre}, it follows that the index $\ell+1$ stays either for the weight $-\varepsilon_m$, if $\lambda_m>0$, or for the weight  $\varepsilon_m$, if $\lambda_m<0$ (since again the indices are given by the decreasing ordering of the translated conformal weights and $\widetilde{w}_{m,+}-\widetilde{w}_{m,-}=2\lambda_m$). Substituting $I=\{\ell+1\}$ in \eqref{Neven2} reduces again the sum to one element and yields the following expression:
\begin{equation}\label{I2even} 
c_{\{\ell+1\}} = 1 - \min_{J\in\NE}\left(\frac{\widetilde{w}_{\ell+1}-\frac{1}{2}}{\widetilde{w}_{\ell+1}-\widetilde{w}_1}\cdot\frac{\prod_{j\in
J} ( \widetilde{w}_{\ell+1} + \widetilde{w}_j)}{\prod_{j\in \widehat J\setminus\{1,\ell+1\}}
( \widetilde{w}_{\ell+1} - \widetilde{w}_j)} \right).
\end{equation}
From the explicit values of the translated conformal weights given by \eqref{confwtrans}, namely: $\widetilde{w}_{m,-}=-\lambda_m+m-\frac{n-1}{2}$ and $\widetilde{w}_{m,+}=\lambda_m-m+\frac{n+1}{2}$, it follows that for $n=2m+1$, as well as for $n=2m$, the term $\left(\widetilde{w}_{\ell+1}-\frac{1}{2}\right)$ is strictly negative, and thus $\frac{\widetilde{w}_{\ell+1}-\frac{1}{2}}{\widetilde{w}_{\ell+1}-\widetilde{w}_1}$ is strictly positive. From the way the sets $J\in\NE$ are defined, by choosing exactly one element from each pair $\{i,2\ell+2-i\}$ for $2\leq i\leq \ell$, it follows that in the product in \eqref{I2even}, there occur only factors of one of the following two types: $\frac{\widetilde{w}_{\ell+1}+\widetilde{w}_i}{\widetilde{w}_{\ell+1}-\widetilde{w}_{2\ell+2-i}}$ or $\frac{\widetilde{w}_{\ell+1}+\widetilde{w}_{2\ell+2-i}}{\widetilde{w}_{\ell+1}-\widetilde{w}_i}$ for some $2\leq i\leq \ell$. From the ordering of the translated conformal weights it turns out that each such factor is strictly positive, showing thus that $c_{\{\ell+1\}}<1$.

3. The ordering of the translated conformal weights implies the following inequalities, for any $i\in\{1,\dots,N\}$, $j\in\{1,\dots,\ell\}$ and $j\neq i,N+2-i$:
\begin{equation*}
\begin{split}
&\frac{\widetilde{w}_i+\widetilde{w}_j}{\widetilde{w}_i-\widetilde{w}_{N+2-j}}>\frac{\widetilde{w}_i+\widetilde{w}_{N+2-j}}{\widetilde{w}_i-\widetilde{w}_{j}}>0,\quad {if }\; i<j \text{ or }\; N+2-j<i,\\
&\frac{\widetilde{w}_i+\widetilde{w}_{N+2-j}}{\widetilde{w}_i-\widetilde{w}_{j}}>\frac{\widetilde{w}_i+\widetilde{w}_j}{\widetilde{w}_i-\widetilde{w}_{N+2-j}}>0,\quad {if}\; j<i<N+2-j.
\end{split}
\end{equation*}

If $N=2\ell-1$, then substituting $I$ in \eqref{Nodd2} with a set formed by a pair of type $I=\{i,N+2-i\}$, with $i\in\{2,\dots,\ell\}$, and using the above relations yields the following expression for the upper bound of the Kato constant:
\begin{equation*}
c_I = 1 - \min \left(\frac{\widetilde{w}_i+ \widetilde{w}_{2\ell+1-i}}{ \widetilde{w}_i- \widetilde{w}_1}, \frac{ \widetilde{w}_i+ \widetilde{w}_{2\ell+1-i}}{ \widetilde{w}_{2\ell+1-i}- \widetilde{w}_1}\right).
\end{equation*}
Similarly, if $N=2\ell$, then substituting $I=\{i,N+2-i\}$ in \eqref{Neven2} yields:
\begin{equation*}
c_I = 1 - \min \left(\frac{(\widetilde{w}_i+ \widetilde{w}_{2\ell+2-i})(\widetilde{w}_i-\frac{1}{2})}{ (\widetilde{w}_i- \widetilde{w}_{\ell+1})(\widetilde{w}_i- \widetilde{w}_{1})}, \frac{ (\widetilde{w}_i+ \widetilde{w}_{2\ell+2-i})(\widetilde{w}_{2\ell+2-i}-\frac{1}{2})}{ (\widetilde{w}_{2\ell+2-i}- \widetilde{w}_{\ell+1})(\widetilde{w}_{2\ell+2-i}- \widetilde{w}_1)}\right).
\end{equation*}
The same argument as in the case 2. shows that $c_I<1$.\hfill $\qed$
\end{proof}

Proposition~\ref{corcgh} proves that all the operators that come up in Branson's classification (listed in Remark~\ref{listminell} in our notation) are elliptic, except for one special case explained in Remark~\ref{spcase}. However, our aim is to determine \emph{all} minimal elliptic operators, so that we still have to eliminate the other possibilities. Namely, on the one hand, we have to show that the generalized gradients corresponding to an element in one of the sets obtained in the case 3. of Proposition~\ref{corcgh} are not elliptic, and on the other hand, that there are no other combinations which provide elliptic operators. Thus, we have to find the maximal non-elliptic operators, in order to conclude that the elliptic operators found in Proposition~\ref{corcgh} are \emph{all} the minimal elliptic operators.

\subsection{Non-elliptic generalized gradients and branching rules}\label{nonellbr} 

The main tool we need here is the branching rule of the special orthogonal group and the following necessary condition for ellipticity (see also \cite{cgh}):

\begin{Lemma}\label{necell} 
Let $P_I:\Gamma(V_\lambda) \to \Gamma(\underset{i\in I}{\oplus} V_i)$ be the operator corresponding to a subset $I$ of $\{1,\dots,N\}$, in the notation introduced by \eqref{descod}. If there exists an irreducible $\SO(n-1)$-subrepresentation of $V_\lambda$ that does not occur as \mbox{$\SO(n-1)$}-subrepresentation of $V_i$ for any $i\in I$, then $P_I$ is not elliptic.
\end{Lemma} 

\begin{proof}
By Definition~\ref{injell}, $P_I$ is (injectively) elliptic if its principal symbol, $\Pi_I:(\mathbb{R}^n)^*\otimes V_\lambda\to\underset{i\in I}{\oplus} V_i$, is injective when restricted to the set of decomposable elements, \emph{i.e.} if for any vector $\alpha\in(\mathbb{R}^n)^*$, $\alpha\neq 0$, the linear map:
\[ V_\lambda \to \underset{i\in I}{\oplus} V_i,  \quad  v \mapsto \Pi_I(\alpha\otimes v)\]
is injective. Since $\SO(n)$ acts transitively on the unit sphere in $(\mathbb{R}^n)^*$, one may, without loss of generality, take $\alpha$ to be a unit vector. Then, the above map is $\SO(n-1)$-equivariant, where $\SO(n-1)$ is the stabilizer group of $\alpha$ under the $\SO(n)$-action on the sphere. The existence of an injective and $\SO(n-1)$-equivariant map between  $V_\lambda$ and $\underset{i\in I}{\oplus} V_i$ shows that any $\SO(n-1)$-subrepresentation of $V_\lambda$ occurs in $V_i$ for some $i\in I$.\hfill $\qed$
\end{proof}

In order to use Lemma~\ref{necell} we have to apply the branching rule for the restriction of an $\SO(n)$-representation to $\SO(n-1)$, which we recall in the sequel (see \emph{e.g.} Theorem 9.16, \cite{knapp}). We consider, as usual, the parametrization of irreducible $\SO(n)$-representations by dominant weights, \emph{i.e.} the weights satisfying the inequalities \eqref{domw}.

\begin{Proposition}[Branching Rule for $\SO(n)$]\label{selectruleso}
\noindent \vspace{-0.1cm}
\begin{description}
	\item[(a)] For the group $\SO(2m+1)$, the irreducible representation with highest weight $\lambda=(\lambda_1,\dots,\lambda_m)$ decomposes with multiplicity $1$ under $\SO(2m)$, and the representations of $\SO(2m)$ that appear are exactly those with highest weights $\gamma=(\gamma_1,\dots,\gamma_m)$ such that
\begin{equation}\label{branchingodd}
\lambda_1\geq\gamma_1\geq\lambda_2\geq\gamma_2\geq\cdots\geq\lambda_{m-1}\geq\gamma_{m-1}\geq\lambda_{m}\geq |\gamma_m|.
\end{equation}	
	\item[(b)] For the group $\SO(2m)$, the irreducible representation with highest weight \linebreak $\lambda=(\lambda_1,\dots,\lambda_m)$ decomposes with multiplicity $1$ under $\SO(2m-1)$, and the representations of $\SO(2m-1)$ that appear are exactly those with highest weights $\gamma=(\gamma_1,\dots,\gamma_{m-1})$ such that
\begin{equation}\label{branchingeven}
\lambda_1\geq\gamma_1\geq\lambda_2\geq\gamma_2\geq\cdots\geq\lambda_{m-1}\geq\gamma_{m-1}\geq |\lambda_m|.
\end{equation}	
\end{description}
\end{Proposition}

From Lemma~\ref{necell} and Proposition~\ref{selectruleso} we obtain the maximal non-elliptic operators as follows:

\begin{Proposition}\label{cormaxne}
The maximal non-elliptic operators $P_J$ are given exactly by the sets $J$ in $\NE$, apart from the special case when $n$ is odd, $N=2\ell-1$ and $\lambda_m\geq 1$. In this case the sets $J$ of $\NE$ that do not contain $\ell$ (which corresponds to the weight $0$) are maximal non-elliptic.
\end{Proposition}

\begin{proof}
We recall that the coordinates of a dominant weight $\lambda$ are given with respect to the basis $\{\varepsilon_i\}_{i=\overline{1,m}}$ introduced in \S ~\ref{sectgengrad}. Here it is more convenient to consider the elements of a set $J$ as weights of the standard representation, instead of the notation with indices corresponding to the ordering of the translated conformal weights.

Let $J$ be a subset in $\NE$, \emph{i.e.} $J$ has cardinality $\ell-1$, where $N=2\ell$ or $N=2\ell-1$. If $n=2m$, then $J$ is obtained by choosing exactly one weight from each pair of relevant weights of type $\{-\varepsilon_i,\varepsilon_{i+1}\}$, for $1\leq i\leq m-2$ and one weight from $\{-\varepsilon_{m-1}, \varepsilon_m\}$, if $\lambda_m>0$, or one weight from $\{-\varepsilon_{m-1}, -\varepsilon_m\}$, if $\lambda_m<0$. If $n=2m+1$, then we consider the sets $J\in \NE$ obtained by choosing exactly one weight from each pair of relevant weights of type $\{-\varepsilon_i,\varepsilon_{i+1}\}$, for $1\leq i\leq m-1$ and the weight $-\varepsilon_m$, if it is relevant. 

For each such set $J$, it is enough to find an $\SO(n-1)$-subrepresentation of $V_\lambda$ that does not occur in $\underset{\varepsilon\in J}{\oplus}V_{\lambda+\varepsilon}$. By Lemma~\ref{necell} it will then follow that the corresponding operator $P_J$ is not elliptic. When enlarging the set $J$ to some set $J'$ by adding any other relevant weight, there is at least one subset $I$ of $J'$ which is equal to one of those listed in Proposition~\ref{corcgh}, showing that $J'$ is elliptic. This means that $J$ is maximal non-elliptic.

For $n=2m$ we choose the irreducible $\SO(2m-1)$-subrepresentation of $\lambda$ with highest weight $\gamma=(\gamma_1,\dots,\gamma_{m-1})$, where the coordinates are defined by the following rule, for each $1\leq i\leq m-2$: 
\begin{equation}
\gamma_i=
\begin{cases}
\lambda_i, & \text{if } \lambda_i=\lambda_{i+1} \text{ or }  -\varepsilon_i\in J\\
\lambda_{i+1}, & \text{if } \varepsilon_{i+1}\in J,
\end{cases}
\end{equation}
and 
\begin{equation}
\gamma_{m-1}=
\begin{cases}
\lambda_{m-1}, & \text{if }  \lambda_{m-1}=\lambda_{m}=0 \text{ or } -\varepsilon_{m-1}\in J\\
\lambda_{m}, & \text{if } \varepsilon_{m}\in J \text{ and } \lambda_m>0\\
-\lambda_{m}, & \text{if } -\varepsilon_{m}\in J \text{ and } \lambda_m<0.
\end{cases}
\end{equation}

We recall that the condition $\lambda_i=\lambda_{i+1}$, for $1\leq i\leq m-2$, is equivalent to the fact that the weights $\{-\varepsilon_i,\varepsilon_{i+1}\}$ are not relevant for $\lambda$ and $\lambda_{m-1}=\lambda_{m}=0$ is the only case when $-\varepsilon_{m-1}$ is not relevant (see \emph{e.g.} Diagram~\ref{figsre}). The coordinates of $\gamma$ fulfill the inequalities \eqref{branchingeven} for the representation $\lambda$, showing that $\gamma$ is an irreducible $\SO(2m-1)$-subrepresentation of $\lambda$. On the other hand, it can be directly checked that the inequalities \eqref{branchingeven} are not satisfied anymore for any of the $\SO(2m)$-representations of highest weight $\lambda+\varepsilon$ with $\varepsilon\in J$, showing that $\gamma$ does not occur as $\SO(2m-1)$-subrepresentation in $\oplus_{\varepsilon\in J} V_{\lambda+\varepsilon}$.

For $n=2m+1$ we similarly choose an irreducible $\SO(2m)$-subrepresentation of $\lambda$ with highest weight $\gamma=(\gamma_1,\dots,\gamma_{m})$, whose coordinates are defined by the following rule, for each $1\leq i\leq m-1$: 
\begin{equation}
\gamma_i=
\begin{cases}
\lambda_i, & \text{if } \lambda_i=\lambda_{i+1} \text{ or }  -\varepsilon_i\in J\\
\lambda_{i+1}, & \text{if } \varepsilon_{i+1}\in J,
\end{cases}
\end{equation}
and $\gamma_m=\lambda_m$. It follows also in this case that the inequalities \eqref{branchingodd} are fulfilled for $\lambda$, but fail for any $\lambda+\varepsilon$ with $\varepsilon\in J$. The branching rule then implies that $\gamma$ is an irreducible $\SO(2m)$-subrepresentation of $V_\lambda$ which does not occur as subrepresentation in $\oplus_{\varepsilon\in J}V_{\lambda+\varepsilon}$. \hfill $\qed$
\end{proof}

\begin{Remark}\label{spcase}
From Proposition~\ref{corcgh} and Corollary~\ref{cormaxne} we recover Branson's classification of minimal elliptic operators, up to an exceptional case. Namely, when $n$ is odd, $N=2\ell-1$ and $\lambda_m>0$, then the zero weight is relevant. If $\lambda$ is moreover properly half-integral, then the corresponding operator $P_{\ell}:V_\lambda M \to V_\lambda M$ is elliptic (by Branson's result), while if $\lambda$ is integral, $P_{\ell}$ is not elliptic. Unfortunately this special case cannot be recovered by the above arguments, since they only involve the translated conformal weights, which are associated to the Lie algebra $\so(n)$, so that they do not distinguish between the groups $\Spin(n)$ and $\SO(n)$. The argument based on the branching rule for establishing the maximal non-elliptic operators does not work either for the zero weight, since in this case the source and target representations are isomorphic.
\end{Remark}

\begin{Remark}\label{othersubgps}
This new approach to the classification of minimal elliptic operators has the advantage that it is mainly based on representation theory and avoids the techniques of harmonic analysis which, as powerful as they are, seem to be specific for the special orthogonal group. Our method suggests that it can be carried over to other subgroups $G$ of the special orthogonal group, in order to provide a similar classification of the minimal elliptic operators obtained from $G$-generalized gradients. We notice that the argument in Lemma~\ref{necell} still works for all the groups $G$ which arise in important geometric situations and are mostly encountered in literature, \emph{i.e.} for those in Berger's list of holonomy groups (which are known to act transitively on the unit sphere) and combined with their branching rules yield a list of non-elliptic operators. Until now we only have partial results, particularly for the group $G_2$, which we shall complete in a forthcoming paper.
\end{Remark}

\end{document}